\documentclass{amsart}

\usepackage{amsfonts,amssymb,amsmath,amsthm,latexsym,booktabs,lscape}

\theoremstyle{definition}
\newtheorem{definition}{Definition}[section]
\newtheorem{remark}[definition]{Remark}

\theoremstyle{plain}
\newtheorem{lemma}[definition]{Lemma}
\newtheorem{proposition}[definition]{Proposition}
\newtheorem{theorem}[definition]{Theorem}

\begin{document}

\title{The Cayley-Dickson Process for Dialgebras}

\author[Felipe-Sosa]{Ra\'ul Felipe-Sosa}

\address{ICIMAF (Instituto de Cibern\'etica, Matem\'atica y F\'isica), Havana, Cuba}

\email{rfs@icimaf.cu}

\author[Felipe]{Ra\'ul Felipe}

\address{CIMAT (Centro de Investigaci\'on en Matem\'aticas), Guanajuato, Mexico}

\email{raulf@cimat.mx}

\author[S\'anchez-Ortega]{Juana S\'anchez-Ortega}

\address{Department of Mathematics and Statistics, York University, Ontario, Canada}

\email{jsanchez@fields.utoronto.ca}

\author[Bremner]{Murray R. Bremner}

\address{Department of Mathematics and Statistics, University of Saskatchewan, Canada}

\email{bremner@math.usask.ca}

\author[Kinyon]{Michael K. Kinyon}

\address{Department of Mathematics, University of Denver, Colorado, USA}

\email{Michael.Kinyon@du.edu}



\maketitle

\begin{abstract}
We adapt the algorithm of Kolesnikov and Pozhidaev, which converts a polynomial identity for
algebras into the corresponding identities for dialgebras, to the Cayley-Dickson
doubling process.  We obtain a generalization of this process to the setting of
dialgebras, establish some of its basic properties, and construct dialgebra analogues 
of the quaternions and octonions.
\end{abstract}


\section{Introduction}

The motivation for this paper is the recent discovery of many new varieties of 
nonassociative structures that can be regarded as ``noncommutative'' analogues of classical 
structures.  This development originated in the work of Bloh \cite{Bloh1,Bloh2} in the 1960's, 
but became much better known after the work of Loday \cite{Loday1,Loday2} 
in the early 1990's on Leibniz algebras; these structures are noncommutative versions of 
Lie algebras which satisfy the Jacobi identity but are not necessarily anticommutative.
The closely related variety of associative dialgebras, also introduced by Loday, provides 
the natural context for the 
universal enveloping algebras of Leibniz algebras: the relation between Leibniz algebras and
associative dialgebras is parallel to the relation between Lie algebras and associative
algebras.  

Ten years after Loday's definition of associative dialgebras, Liu \cite{Liu} introduced 
alternative dialgebras, the natural 
analogue of alternative algebras in the setting of structures with two operations.  Shortly 
after that, Felipe and Vel\'asquez \cite{Felipe} initiated the study of quasi-Jordan algebras  
(Jordan dialgebras), which are related to Jordan algebras as Leibniz algebras 
are to Lie algebras.  Around the same time, Kolesnikov \cite{Kolesnikov} developed a general 
method for passing from a variety of nonassociative algebras defined by polynomial identities 
to the corresponding variety of dialgebras.  This method has been simplified and formalized 
in the so-called KP algorithm \cite{BFSO}, which is a concrete realization of the white Manin 
product introduced by Vallette \cite{Vallette} by the permutad \textsf{Perm} defined by 
Chapoton \cite{Chapoton} as the Koszul dual of the operad \textsf{preLie}.  For further information on
these developments, see \cite{Baietal,BaiLiuNi,BremnerSurvey,HouBai,KolesnikovVoronin,NiBai}.

An important topic in classical algebra is the theory of composition algebras and their 
connection with quadratic forms and the eight-square theorem.
This leads to the construction, starting from the real numbers, of the complex numbers, 
quaternions, and octonions, through a doubling process which originated in the works of 
Cayley and Dickson.  For the early history of these developments, see Dickson \cite{Dickson}, 
and for the completion of the classical theory, see Albert \cite{Albert} and Schafer \cite{Schafer}.
The purpose of the present work is to determine the natural generalization of the Cayley-Dickson
process to the setting of dialgebras.


\section{Preliminaries}

\subsection{The Cayley-Dickson (CD) process}

We recall the Cayley-Dickson doubling process from Zhevlakov et al.
\cite[\S 2.2]{Zhevlakov}; see also Baez \cite[\S 2.2]{Baez}, Albuquerque and Majid \cite[\S 4]{AM}.
Let $A$ be a unital algebra over the field $\mathbb{F}$ with an involution $a \mapsto a^\ast$:
that is, a linear operator on $A$ for which $(a^\ast)^\ast = a$, $(ab)^\ast = b^\ast a^\ast$
and $a+a^\ast, a a^\ast \in \mathbb{F}$ for all $a, b \in A$.
Given $\gamma \in \mathbb{F}$, $\gamma \ne 0$, we define a bilinear product on the vector
space direct sum $A \oplus A$ as follows:
  \begin{equation}
  \label{generalgamma}
  (a,b) (c,d)  =  ( a c + \gamma d b^\ast, a^\ast d + c b ).
  \end{equation}
The identity element of $A \oplus A$ is $(1,0)$, the subspace $A \oplus \{0\}$
is a subalgebra isomorphic to $A$, and $(0,1)^2 = (\gamma,0)$; 
the endomorphism $(a,b) \mapsto (a^\ast,-b)$ is an involution.
In the special case $\gamma = -1$ we obtain
  \begin{equation}
  \label{cdprocess}
  (a,b) (c,d)  =  ( a c - d b^\ast, a^\ast d + c b ).
  \end{equation}
Starting with $A = \mathbb{R}$ (real numbers), the identity involution $a^\ast = a$,
and using $\gamma = -1$,
we obtain successively the algebras $\mathbb{C}$ (complex numbers), $\mathbb{H}$ (quaternions), 
and $\mathbb{O}$ (octonions); the process continues but only the first four 
are division algebras. 

\subsection{The Kolesnikov-Pozhidaev (KP) algorithm}

This gives a generalization to dialgebras of an arbitrary variety of 
nonassociative structures by converting a multilinear polynomial identity of degree $d$ for 
a single $n$-ary operation into a family of $d$ multilinear identities of degree $d$ for $n$ 
new $n$-ary operations.
This algorithm was introduced by Kolesnikov \cite{Kolesnikov} for identities in a binary operation,
and generalized by Pozhidaev \cite{P2} to $n$-ary operations.
We state the algorithm in general, but for our purposes we require only the 
binary version.
We consider a multilinear $n$-ary operation $\{-,\dots,-\}$, and introduce $n$ new $n$-ary operations 
$\{-,\dots,-\}_j$ distinguished by subscripts $j = 1, \dots, n$.

First, we introduce the following \textbf{0-identities} for $i, j = 1, \dots, n$ with $i \ne j$ 
and $k, \ell = 1, \dots, n$; these identities say that the new operations are interchangeable 
in argument $i$ of operation $j$ when $i \ne j$:
  \allowdisplaybreaks
  \begin{align*}
  &
  \{ a_1, \dots, a_{i-1}, \{ b_1, \cdots, b_n \}_k, a_{i+1}, \dots, a_n \}_j
  \equiv
  \\
  &
  \{ a_1, \dots, a_{i-1}, \{ b_1, \cdots, b_n \}_\ell, a_{i+1}, \dots, a_n \}_j.
  \end{align*}
  
Second, we consider a multilinear identity $I(a_1,\dots,a_d)$ of degree $d$ in the $n$-ary operation $\{-,\dots,-\}$.
We apply the following rule to each monomial of the identity; let $a_{\pi(1)} a_{\pi(2)} \dots a_{\pi(d)}$ 
be such a monomial with some placement of operation symbols where $\pi$ is a permutation of $1,\dots,d$.
For $i = 1, \dots, d$ we convert this monomial into a new monomial of the same degree in the $n$ new operations 
according to the position of the central argument $a_i$.
For each occurrence of the original operation, we have the following cases:
  \begin{itemize}
  \item
  If $a_i$ occurs in argument $j$ then $\{\dots\}$ becomes $\{\dots\}_j$.
  \item
  If $a_i$ does not occur in any argument then
    \begin{itemize}
    \item
    if $a_i$ occurs to the left of the original operation, $\{\dots\}$ becomes $\{\dots\}_1$,
    \item
    if $a_i$ occurs to the right of the original operation, $\{\dots\}$ becomes $\{\dots\}_n$.
    \end{itemize}
  \end{itemize}
The resulting new identity is called the \textbf{KP identity} corresponding to 
$I(a_1,\dots,a_d)$.

The choice of new operations, in the two subcases under the second bullet above,
gives a convenient normal form for the monomial: by the 0-identities, the subscripts 1 and $n$ 
could be replaced by any other subscripts.  Suppose that $a_i$ is the central
argument and that the identity $I(a_1,\dots,a_d)$ contains a monomial of this form:
  \[
  \{ \dots, 
  \overbrace{\{ -, \dots, - \}}^{\text{argument $i$}}, 
  \dots, 
  \overbrace{\{ \dots, a_i, \dots \}}^{\text{argument $j$}}, 
  \dots, 
  \overbrace{\{ -, \dots, - \}}^{\text{argument $k$}},
  \dots \}.
  \]
Since $a_i$ occurs in argument $j$, the outermost operation must receive subscript $j$:
  \[
  \{ \dots, \{ -, \dots, - \}, \dots, \{ \dots, a_i, \dots \}, \dots, \{ -, \dots, - \}, \dots \}_j.
  \]
Our convention above attaches subscripts $n$ and 1 to arguments $i$ and $k$ respectively:
  \[
  \{ \dots, \{ -, \dots, - \}_n, \dots, \{ \dots, a_i, \dots \}, \dots, \{ -, \dots, - \}_1, \dots \}_j.
  \]
Since these subscripts occur in arguments $i \ne j$ and $k \ne j$ of operation $j$, 
the 0-identities imply that any other subscripts would give an equivalent identity.

\subsection{Associative and nonassociative dialgebras}

We recall some basic definitions and examples.

\begin{definition}
A \textbf{dialgebra} is a vector space $A$ with two bilinear operations $A \times A \to A$,
denoted $a \dashv b$ and $a \vdash b$, and called the \textbf{left} and \textbf{right} products.
\end{definition}

\begin{definition} \label{cd5}
If we start with one binary operation but assume no polynomial identities, the 0-identities are
the \textbf{left and right bar identities},
  \begin{equation}
  \label{baridentities}
  a \dashv ( b \dashv c ) \equiv a \dashv ( b \vdash c ),
  \qquad
  ( a \dashv b ) \vdash c \equiv ( a \vdash b ) \vdash c,
  \end{equation}
which define a \textbf{0-dialgebra} (Loday \cite{Loday2, Loday3}, Kolesnikov \cite{Kolesnikov}).
\end{definition}

\begin{definition}
From the commutative identity $\{a,b\} \equiv \{b,a\}$ we obtain the KP identities
$\{a,b\}_1 \equiv \{b,a\}_2$ and $\{a,b\}_2 \equiv \{b,a\}_1$;
in standard notation, both are equivalent to $a \dashv b \equiv b \vdash a$.
Combining this with the bar identities, we obtain the definition of a
\textbf{commutative 0-dialgebra}.
\end{definition}

\begin{definition}
\cite{Loday2}
Applying the KP algorithm to associativity $\{\{a,b\},c\} \equiv \{a,\{b,c\}\}$ gives 
the definition of an \textbf{associative 0-dialgebra}: a 0-dialgebra which
satisfies \textbf{left, inner} and \textbf{right associativity}:
  \begin{equation}
  ( a \dashv b ) \dashv c \equiv a \dashv ( b \dashv c ), \;\; 
  ( a \vdash b ) \dashv c \equiv a \vdash ( b \dashv c ), \;\;
  ( a \vdash b ) \vdash c \equiv a \vdash ( b \vdash c ).
  \end{equation}
\end{definition}

Calculating the KP identities corresponding to left and right alternativity,
  \[
  (ab)c - a(bc) + (ba)c - b(ac) \equiv 0,
  \qquad
  (ab)c - a(bc) + (ac)b - a(cb) \equiv 0,
  \]
and eliminating redundant identities, gives the following three identities:
  \allowdisplaybreaks
  \begin{align*}
  &
  ( a \dashv b ) \dashv c - a \dashv ( b \dashv c )
  +
  ( c \vdash b ) \vdash a - c \vdash ( b \vdash a )
  \equiv 0,
  \\
  &
  ( a \dashv b ) \dashv c - a \dashv ( b \dashv c )
  -
  ( b \vdash c ) \vdash a + b \vdash ( c \vdash a )
  \equiv 0,
  \\
  &
  ( a \vdash b ) \dashv c - a \vdash ( b \dashv c )
  +
  ( a \vdash c ) \vdash b - a \vdash ( c \vdash b )
  \equiv 0.
  \end{align*}
  
\begin{definition}
\label{alternativedialgebra}
\cite{Liu}
An \textbf{alternative 0-dialgebra} is a 0-dialgebra satisfying
  \[
  (a,b,c)_\dashv + (c,b,a)_\vdash \equiv 0,
  \quad
  (a,b,c)_\dashv - (b,c,a)_\vdash \equiv 0,
  \quad
  (a,b,c)_\times + (a,c,b)_\vdash \equiv 0,  
  \]
where the left, inner and right associators are
  \allowdisplaybreaks
  \begin{align*}
  (a,b,c)_\dashv &= ( a \dashv b ) \dashv c - a \dashv ( b \dashv c ),
  \\
  (a,b,c)_\times &= ( a \vdash b ) \dashv c - a \vdash ( b \dashv c ),
  \\
  (a,b,c)_\vdash &= ( a \vdash b ) \vdash c - a \vdash ( b \vdash c ),  
  \end{align*}
Every associative 0-dialgebra is an alternative 0-dialgebra. 
\end{definition}


\section{Basic properties of the doubling process for dialgebras}

We start with the classical Cayley-Dickson process of equation \eqref{generalgamma}.
To simplify the calculations, we assume $\gamma = -1$, and use equation \eqref{cdprocess};
the general case can be worked out similarly.  
We apply the KP algorithm to equation \eqref{cdprocess}; 
this is a non-standard application since \eqref{cdprocess} is not,
strictly speaking, a polynomial identity.
  
\begin{definition}
In equation \eqref{cdprocess}, we first regard $(a,b)$ as central and then $(c,d)$.  
We obtain the following left and right products on the vector space $A \oplus A$:
  \allowdisplaybreaks
  \begin{align}
  (a,b) \dashv (c,d) &= ( a \dashv c - d \vdash b^\ast, a^\ast \dashv d + c \vdash b ),
  \label{CDdialgebra1}
  \\
  (a,b) \vdash (c,d) &= ( a \vdash c - d \dashv b^\ast, a^\ast \vdash d + c \dashv b ).
  \label{CDdialgebra2}
  \end{align}
Since the involution is a unary operation, we extend it as in the algebra case:
  \begin{equation}
  (a,b)^\ast = (a^\ast,-b).
  \label{CDdialgebra3}
  \end{equation}
Equations \eqref{CDdialgebra1}--\eqref{CDdialgebra3} define the \textbf{Cayley-Dickson process for dialgebras}.
The dialgebra $A \oplus A$ with these operations is the \textbf{Cayley-Dickson double} of $A$.
\end{definition}

\begin{remark}
There is another convention for the Cayley-Dickson construction. 
For an algebra $A$ with involution $\ast$ and $\gamma\in \mathbb{F}$, 
define a bilinear multiplication by
\[
(a,b)\diamond (c,d) = (ac + \gamma d^* b, bc^* + da )\,.
\]
Denoting our Cayley-Dickson multiplication by $\cdot$, 
there is an isomorphism between $(A \oplus A,\cdot)$ and $(A \oplus A,\diamond)$ given by $(a,b) \mapsto (a,b^*)$.
We show that our definition of the Cayley-Dickson process for dialgebras
does not depend on our convention for the multiplication in the doubled algebra. 
If we apply the KP algorithm using $\diamond$, we obtain the following two operations in $A\oplus A$:
  \begin{align*}
  (a,b) \dashv_{\diamond} (c,d)  
  &=
  (a\dashv c + \gamma (d^* \vdash b), b\dashv c^* + d\vdash a) 
  \\
  (a,b) \vdash_{\diamond} (c,d) 
  &=
  (a\vdash c + \gamma (d^* \dashv b), b\vdash c^* + d\dashv a)
  \end{align*}
The mapping $\phi\colon A \oplus A\to A \oplus A$ defined by $(a,b)\mapsto (a,b^*)$ gives an isomorphism of 
$(A \oplus A,\dashv,\vdash)$ with $(A \oplus A,\dashv_{\diamond},\vdash_{\diamond})$.
Indeed we compute 
  \begin{align*}
  &
  \phi( (a,b)\vdash (c,d) ) 
  =
  ( a\vdash c + \gamma d\dashv b^*, ( a^*\vdash d + c\dashv b)^* ) 
  \\
  &= 
  ( a\vdash c + \gamma d\dashv b^*, d^*\dashv a + b^*\vdash c^* ) 
  = 
  (a,b^*)\vdash_{\diamond} (c,d^*) 
  = 
  \phi(a,b)\vdash_{\diamond} \phi(c,d),
\end{align*}
and the calculation for $\dashv$ and $\dashv_{\diamond}$ is similar.
\end{remark}

\begin{definition} \label{dialgebrainvolution}
\cite{P1}
From the involution identity $(ab)^\ast \equiv b^\ast a^\ast$ we obtain the KP identities
$( a \dashv b )^\ast \equiv b^\ast \vdash a^\ast$ and
$( a \vdash b )^\ast \equiv b^\ast \dashv a^\ast$.
A \textbf{dialgebra with involution} satisfies these identities and $(a^\ast)^\ast \equiv a$.
\end{definition}

\begin{definition}
Let $A$ be a 0-dialgebra.
An element $e \in A$ is a \textbf{bar unit} if $a \dashv e = a = e \vdash a$ for all $a \in A$. 
(If $A$ has an involution and $e$ is a bar unit then so is $e^\ast$.)
An element $e \in A$ is a \textbf{bar zero} if $a \dashv e = 0 = e \vdash a$ for all $a \in A$. 
\end{definition}

\begin{lemma} \label{dibar}
Let $A$ be a dialgebra with involution.  Then $A$ satisfies the left bar identity 
if and only if it satisfies the right bar identity. 
\end{lemma}

\begin{proof}
If $A$ satisfies the left bar identity then it satisfies the right bar identity:
\[
a \dashv ( b \dashv c )
\equiv
(
( c^\ast \vdash b^\ast ) \vdash a^\ast
)^\ast
\equiv
(
( c^\ast \dashv b^\ast ) \vdash a^\ast
)^\ast
\equiv
a \dashv ( b \vdash c ).
\]
The converse is similar.
\end{proof}

\begin{proposition} \label{AAinvolution}
If $A$ is a 0-dialgebra with involution then equations \eqref{CDdialgebra1}--\eqref{CDdialgebra3}
make $A \oplus A$ into a 0-dialgebra with involution.
\end{proposition}

\begin{proof}
We first show that $A \oplus A$ satisfies the involution identities:
  \begin{align*}
  &
  \big( (a,b)^\ast \big)^\ast 
  = 
  (a^\ast,-b)^\ast = \big( (a^\ast)^\ast, -(-b) \big) = (a,b),
  \\
  &
  \big( (a,b) \dashv (c,d) \big)^\ast
  =
  \big( a \dashv c - d \vdash b^\ast, a^\ast \dashv d + c \vdash b \big)^\ast
  \\
  &=
  \big( ( a \dashv c - d \vdash b^\ast )^\ast, -( a^\ast \dashv d + c \vdash b ) \big)
  =
  ( c^\ast \vdash a^\ast - b \dashv d^\ast, - c \vdash b - a^\ast \dashv d )
  \\
  &=
  (c^\ast,-d) \vdash (a^\ast,-b)
  =  
  (c,d)^\ast \vdash (a,b)^\ast,
  \end{align*}
and the calculation for $\big( (a,b) \vdash (c,d) \big)^\ast$ is similar.
To prove the left bar identity:
  \allowdisplaybreaks
  \begin{align*}
  &
  (a,b) \dashv \big( (c,d) \dashv (e,f) \big)
  =
  (a,b) \dashv ( c \dashv e - f \vdash d^\ast, c^\ast \dashv f + e \vdash d )
  \\
  &=
  \big( \,
  a \dashv ( c \dashv e - f \vdash d^\ast ) - ( c^\ast \dashv f + e \vdash d ) \vdash b^\ast, 
  \\
  &\quad\quad\quad
  a^\ast \dashv ( c^\ast \dashv f + e \vdash d ) + ( c \dashv e - f \vdash d^\ast ) \vdash b 
  \, \big)
  \\
  &=
  \big( \,
  a \dashv ( c \vdash e - f \dashv d^\ast ) - ( c^\ast \vdash f + e \dashv d ) \vdash b^\ast, 
  \\
  &\quad\quad\quad
  a^\ast \dashv ( c^\ast \vdash f + e \dashv d ) + ( c \vdash e - f \dashv d^\ast ) \vdash b 
  \, \big) 
  \\
  &=
  (a,b) \dashv ( c \vdash e - f \dashv d^\ast, c^\ast \vdash f + e \dashv d )
  =
  (a,b) \dashv ( (c,d) \vdash (e,f) ).
  \end{align*}
The right bar identity now follows from Lemma \ref{dibar}.
\end{proof}

\begin{remark}
From any 0-dialgebra one obtains in a canonical way a corresponding algebra by forming 
the quotient modulo the ideal generated by all elements of the form $a \dashv b - a \vdash b$.
We will show that this functor commutes with the Cayley-Dickson process for dialgebras.
Let $A$ be a $0$-dialgebra. Let $I_A$ denote the smallest ideal of $A$ such that 
$A_{\mathrm{alg}} = A/I_A$ is an algebra in which the operations $\vdash$ and $\dashv$ coincide. 
The natural surjective homomorphism $A\to A_{\mathrm{alg}}$ defined by $a\mapsto a + I_A$ 
gives a functor from the category of 0-dialgebras to the category of algebras (over the same field). 
It is well known and easy to see that $I_A$ is generated by all elements of $A$ of the form $a\vdash b - a\dashv b$.
Let $A$ be a $0$-dialgebra and let $A\oplus A$ be its Cayley-Dickson double. 
We will show that $I_{A\oplus A} = I_A \oplus I_A$ and hence
$(A\oplus A)_{\mathrm{alg}} \cong A_{\mathrm{alg}}\oplus A_{\mathrm{alg}}$.
The generators of $I_{A\oplus A}$ consist all elements of the form
  \begin{align*}
  &
  (a,b)\vdash (c,d) - (a,b)\dashv (c,d) 
  \\
  &= 
  (a\vdash b + \gamma ( d\dashv b^* ), a^*\vdash d + c\dashv b) 
  - (a\dashv b + \gamma d\vdash b^*, a^*\dashv d + c\vdash b) 
  \\
  &= 
  (a\vdash b -  a\dashv b + \gamma ( d\dashv b^* - d\vdash b^*), 
  a^*\vdash d - a^*\dashv d + c\dashv b - c\vdash b)\,.
  \end{align*}
Since each component is evidently contained in $I_A$, we see that $I_{A\oplus A}\subseteq I_A\oplus I_A$.
On the other hand, for $a\vdash b - a\dashv b \in I_A$, we have
  \[
  (a\vdash b - a\dashv b,0) = (a,b)\vdash (0,0) - (a,b)\dashv (0,0)\in I_{A\oplus A},
  \] 
so that $I_A\oplus \{0\}\subseteq I_{A\oplus A}$. 
Similarly $\{0\}\oplus I_A\subseteq I_{A\oplus A}$. 
Thus $I_A\oplus I_A\subseteq I_{A\oplus A}$.
For the remaining assertion, it is straightforward to check that 
$\phi\colon A\oplus A\to A_{\mathrm{alg}}\oplus A_{\mathrm{alg}}$ given by $\phi(a,b) = (a + I_A, b + I_A)$ 
is a homomorphism where both $A\oplus A$ and $A_{\mathrm{alg}}\oplus A_{\mathrm{alg}}$ 
have their Cayley-Dickson structures. Clearly $\phi$ is surjective and by the first claim,
$\ker(\phi) = I_A\oplus I_A = I_{A\oplus A}$. 
Therefore $(A\oplus A)_{\mathrm{alg}} \cong A_{\mathrm{alg}}\oplus A_{\mathrm{alg}}$.
\end{remark}


\section{From commutative dialgebras to associative dialgebras}

\begin{lemma} \label{LRassociators}
If $A$ is a dialgebra with involution, then the right associator can be expressed in terms of the
left associator by the equation
  \[
  ( x, y, z )_\vdash \equiv - \big( ( z^\ast, y^\ast, x^\ast )_\dashv \big)^\ast.
  \]  
\end{lemma}

\begin{proof}
Straightforward calculation.
\end{proof}

\begin{theorem}
If $A$ is a commutative associative 0-dialgebra with involution, 
then $A \oplus A$ is an associative 0-dialgebra with involution.
\end{theorem}
 
\begin{proof}
We first prove that $A \oplus A$ satisfies left associativity.
Applying equation \eqref{CDdialgebra1} twice, and using the involution and bilinearity, 
we obtain
  \allowdisplaybreaks
  \begin{align*}
  &
  \big( (a,b) \dashv (c,d) \big) \dashv (e,f)
  =
  \\
  &\quad
  \big( \,
  ( a \dashv c ) \dashv e - ( d \vdash b^\ast ) \dashv e - 
  f \vdash ( d^\ast \vdash a ) - f \vdash ( b^\ast \dashv c^\ast ), 
  \\
  &\quad\quad
  ( c^\ast \vdash a^\ast ) \dashv f - ( b \dashv d^\ast ) \dashv f + 
  e \vdash ( a^\ast \dashv d ) + e \vdash ( c \vdash b ) 
  \, \big).
  \end{align*}
We now apply the following equations which use the assumptions on $A$:
  \allowdisplaybreaks
  \begin{align*}
  ( a \dashv c ) \dashv e
  &\equiv
  a \dashv ( c \dashv e ),
  \\
  ( d \vdash b^\ast ) \dashv e
  &\equiv
  e \vdash ( d \vdash b^\ast )
  \equiv
  ( e \vdash d ) \vdash b^\ast,
  \\
  f \vdash ( d^\ast \vdash a ) 
  &\equiv
  ( f \vdash d^\ast ) \vdash a
  \equiv
  a \dashv ( f \vdash d^\ast ),
  \\
  f \vdash ( b^\ast \dashv c^\ast )
  &\equiv
  f \vdash ( c^\ast \vdash b^\ast )
  \equiv
  ( f \vdash c^\ast ) \vdash b^\ast
  \equiv
  ( c^\ast \dashv f ) \vdash b^\ast,
  \\
  ( c^\ast \vdash a^\ast ) \dashv f
  &\equiv
  ( a^\ast \dashv c^\ast ) \dashv f
  \equiv
  a^\ast \dashv ( c^\ast \dashv f ),
  \\
  ( b \dashv d^\ast ) \dashv f
  &\equiv
  b \dashv ( d^\ast \dashv f )
  \equiv
  ( f \vdash d^\ast ) \vdash b,
  \\
  e \vdash ( a^\ast \dashv d )
  &\equiv
  e \vdash ( d \vdash a^\ast )
  \equiv
  ( e \vdash d ) \vdash a^\ast 
  \equiv
  a^\ast \dashv ( e \vdash d ),
  \\
  e \vdash ( c \vdash b ) 
  &\equiv
  ( e \vdash c ) \vdash b
  \equiv
  ( c \dashv e ) \vdash b.
  \end{align*}
From these we obtain
  \allowdisplaybreaks
  \begin{align*}
  &
  \big( (a,b) \dashv (c,d) \big) \dashv (e,f)
  =
  \\
  &\quad
  \big( \,
  a \dashv ( c \dashv e ) - a \dashv ( f \vdash d^\ast ) - 
  ( c^\ast \dashv f ) \vdash b^\ast - ( e \vdash d ) \vdash b^\ast, 
  \\
  &\quad\quad
  a^\ast \dashv ( c^\ast \dashv f ) + a^\ast \dashv ( e \vdash d ) + 
  ( c \dashv e ) \vdash b - ( f \vdash d^\ast ) \vdash b 
  \, \big) .
  \end{align*}
Applying equation \eqref{CDdialgebra1} and bilinearity, this equals
$(a,b) \dashv \big( (c,d) \dashv (e,f) \big)$.

Right associativity now follows from Lemma \ref{LRassociators}.

For inner associativity, we proceed as above and obtain
  \allowdisplaybreaks
  \begin{align*}
  &
  \big( (a,b) \vdash (c,d) \big) \dashv (e,f) 
  =
  \\
  &\quad
  \big( \,
  ( a \vdash c ) \dashv e - ( d \dashv b^\ast ) \dashv e - 
  f \vdash ( d^\ast \dashv a ) - f \vdash ( b^\ast \vdash c^\ast ),
  \\
  &\quad\quad
  ( c^\ast \dashv a^\ast ) \dashv f - ( b \vdash d^\ast ) \dashv f + 
  e \vdash ( a^\ast \vdash d ) + e \vdash ( c \dashv b )
  \, \big). 
  \end{align*}
We now apply the following equations which use the assumptions on $A$:
  \allowdisplaybreaks
  \begin{align*}
  ( a \vdash c ) \dashv e
  &\equiv
  a \vdash ( c \dashv e ),
  \\
  ( d \dashv b^\ast ) \dashv e
  &\equiv
  ( b^\ast \vdash d ) \dashv e
  \equiv
  b^\ast \vdash ( d \dashv e )
  \equiv
  ( e \vdash d ) \dashv b^\ast,
  \\
  f \vdash ( d^\ast \dashv a )
  &\equiv
  ( f \vdash d^\ast ) \dashv a
  \equiv
  a \vdash ( f \vdash d^\ast ),
  \\
  f \vdash ( b^\ast \vdash c^\ast )
  &\equiv
  f \vdash ( c^\ast \dashv b^\ast )
  \equiv
  ( f \vdash c^\ast ) \dashv b^\ast
  \equiv
  ( c^\ast \dashv f ) \dashv b^\ast,
  \\
  ( c^\ast \dashv a^\ast ) \dashv f
  &\equiv
  ( a^\ast \vdash c^\ast ) \dashv f
  \equiv
  a^\ast \vdash ( c^\ast \dashv f ),
  \\
  ( b \vdash d^\ast ) \dashv f
  &\equiv
  b \vdash ( d^\ast \dashv f )
  \equiv
  ( f \vdash d^\ast ) \dashv b,
  \\
  e \vdash ( a^\ast \vdash d )
  &\equiv
  e \vdash ( d \dashv a^\ast )
  \equiv
  ( e \vdash d ) \dashv a^\ast
  \equiv
  a^\ast \vdash ( e \vdash d ),
  \\
  e \vdash ( c \dashv b )
  &\equiv
  ( e \vdash c ) \dashv b 
  \equiv
  ( c \dashv e ) \dashv b.
  \end{align*}  
From these we obtain
  \allowdisplaybreaks
  \begin{align*}
  &
  \big( (a,b) \vdash (c,d) \big) \dashv (e,f)
  =
  \\
  &\quad
  \big( \,
  a \vdash ( c \dashv e ) - a \vdash ( f \vdash d^\ast ) - 
  ( c^\ast \dashv f ) \dashv b^\ast - (e \vdash d ) \dashv b^\ast, 
  \\
  &\quad\quad
  a^\ast \vdash ( c^\ast \dashv f ) + a^\ast \vdash ( e \vdash d ) + 
  ( c \dashv e ) \dashv b - ( f \vdash d^\ast ) \dashv b 
  \, \big).
  \end{align*}
Applying equation \eqref{CDdialgebra1} and bilinearity, this equals
$(a,b) \vdash \big( (c,d) \dashv (e,f) \big)$.

Now Proposition \ref{AAinvolution} completes the proof.
\end{proof}

\begin{remark}
If $A$ is a commutative associative 0-dialgebra with involution then
  \[
  (a,b) \dashv (c,d) - (c,d) \vdash (a,b)
  =
  \big( \,
  b \dashv d^\ast - ( b \dashv d^\ast)^\ast, \,
  ( a^\ast {-} a ) \dashv d + ( c {-} c^\ast ) \vdash b
  \, \big).
  \]
Thus $A \oplus A$ is not necessarily commutative.
\end{remark}  


\section{From associative dialgebras to alternative dialgebras}

To prove the next theorem, we need to impose further conditions on a 0-dialgebra with involution,
obtained by applying the KP algorithm to the assumption that symmetric elements
in an algebra with involution commute with every element.  
Assume that $A$ is an algebra with involution and $(x+x^\ast)y \equiv y(x+x^\ast)$ for all $x, y \in A$.  
If we make $x$ (respectively $y$) the central variable then we obtain
  \[
  ( x + x^\ast ) \dashv y \equiv y \vdash ( x + x^\ast ),
  \qquad
  ( x + x^\ast ) \vdash y \equiv y \dashv ( x + x^\ast ).
  \]  

\begin{definition}
Let $A$ be a 0-dialgebra with involution.  We introduce the notation 
$\mathrm{sym}(x) = x + x^\ast$, and we write $\{ x, y \} = x \dashv y - y \vdash x$ for the Leibniz bracket.
We say that $A$ is a \textbf{partially symmetric} 0-dialgebra if it satisfies the identities
  \[
  \{ \mathrm{sym}(x), y \} \equiv 0,
  \qquad
  \{ x, \mathrm{sym}(y) \} \equiv 0.
  \]
\end{definition}

\begin{lemma} \label{movestar}
In every partially symmetric 0-dialgebra we have
  \[
  \mathrm{sym}( x \dashv y^\ast ) \equiv \mathrm{sym}( x^\ast \dashv y ), 
  \quad \quad 
  \mathrm{sym}( x \dashv y ) \equiv \mathrm{sym}( x^\ast \dashv y^\ast ).
  \]
\end{lemma}

\begin{proof}
Using the definition of partially symmetric 0-dialgebra we have
  \begin{align*}
  \mathrm{sym}( x \dashv y^\ast )
  &\equiv
  x \dashv y^\ast + y \vdash x^\ast
  =
  x \dashv ( \mathrm{sym}(y) - y ) + y \vdash ( \mathrm{sym}(x) - x )
  \\
  &=
  x \dashv \mathrm{sym}(y) - x \dashv y + y \vdash \mathrm{sym}(x) - y \vdash x
  \\
  &=
  \mathrm{sym}(y) \vdash x - x \dashv y + \mathrm{sym}(x) \dashv y - y \vdash x
  \\
  &=
  \mathrm{sym}(y) \vdash x - y \vdash x + \mathrm{sym}(x) \dashv y - x \dashv y
  \\
  &=
  ( \mathrm{sym}(y) - y ) \vdash x + ( \mathrm{sym}(x) - x ) \dashv y
  \\
  &=
  y^\ast \vdash x + x^\ast \dashv y
  =
  x^\ast \dashv y + y^\ast \vdash x
  =
  \mathrm{sym}( x^\ast \dashv y ).
  \end{align*}
The second identity is a consequence of the first.
\end{proof}

\begin{lemma}
In every partially symmetric 0-dialgebra we have $\mathrm{sym}( \{ x, y \} ) \equiv 0$.
\end{lemma}

\begin{proof}
Using the definition of partially symmetric 0-dialgebra we have
  \begin{align*}
  \{ x^\ast, y \}
  &=
  \{ \mathrm{sym}(x) - x, y \}
  =
  \{ \mathrm{sym}(x), y \} - \{ x, y \}
  \equiv
  - \{ x, y \},
  \\
  \{ x, y^\ast \}
  &=
  \{ x, \mathrm{sym}(y) - y \}
  =
  \{ x, \mathrm{sym}(y) \} - \{ x, y \}
  \equiv
  - \{ x, y \}.
  \end{align*}
Thus we can replace stars inside a Leibniz bracket by minus signs.  Therefore
  \[
  \{ x, y \}^\ast
  =
  ( x \dashv y - y \vdash x )^\ast
  =
  y^\ast \vdash x^\ast - x^\ast \dashv y^\ast
  =
  - \{ x^\ast, y^\ast \}
  \equiv
  \{ x, y^\ast \}
  \equiv
  - \{ x, y \},
  \]
and hence
$\mathrm{sym}( \{ x, y \} ) = \{x,y\} + \{x,y\}^\ast = 0$.
\end{proof}

\begin{lemma}
If $A$ is a partially symmetric 0-dialgebra then so is its Cayley-Dickson double $A \oplus A$.
\end{lemma}

\begin{proof}
Straightforward calculation.
\end{proof}

The next sequence of lemmas can be verified by elementary calculations.

\begin{lemma} \label{oldLemma5.8}
Let $A$ be a 0-dialgebra with involution.  
In the Cayley-Dickson double $A \oplus A$, the left, inner and right associators are as follows:
  \begin{align*}
  &
  \big( \, (a,b), \, (c,d), \, (e,f) \, \big)_\dashv
  =
  \\
  &\quad
  \big(
  ( a \dashv c ) \dashv e
  - ( d \vdash b^\ast ) \dashv e 
  - f \vdash ( d^\ast \vdash a ) 
  - f \vdash ( b^\ast \dashv c^\ast )
  \\
  &\qquad
  - a \dashv ( c \dashv e )
  + a \dashv ( f \vdash d^\ast )
  + ( c^\ast \dashv f ) \vdash b^\ast
  + ( e \vdash d ) \vdash b^\ast,
  \\
  &\qquad
  ( c^\ast \vdash a^\ast ) \dashv f
  - ( b \dashv d^\ast ) \dashv f
  + e \vdash ( a^\ast \dashv d )
  + e \vdash ( c \vdash b )
  \\
  &\qquad
  - a^\ast \dashv ( c^\ast \dashv f )
  - a^\ast \dashv ( e \vdash d )
  - ( c \dashv e ) \vdash b
  + ( f \vdash d^\ast ) \vdash b
  \big),
  \\
  &
  \big( \, (a,b), \, (c,d), \, (e,f) \, \big)_\times
  =
  \\
  &\quad
  \big(
  ( a \vdash c ) \dashv e
  - ( d \dashv b^\ast ) \dashv e
  - f \vdash ( d^\ast \dashv a )
  - f \vdash ( b^\ast \vdash c^\ast )
  \\
  &\qquad
  - a \vdash ( c \dashv e )
  + a \vdash ( f \vdash d^\ast )
  + ( c^\ast \dashv f ) \dashv b^\ast
  + ( e \vdash d ) \dashv b^\ast,
  \\
  &\qquad
  ( c^\ast \dashv a^\ast ) \dashv f
  - ( b \vdash d^\ast ) \dashv f
  + e \vdash ( a^\ast \vdash d )
  + e \vdash ( c \dashv b )
  \\
  &\qquad
  - a^\ast \vdash ( c^\ast \dashv f )
  - a^\ast \vdash ( e \vdash d )
  - ( c \dashv e ) \dashv b
  + ( f \vdash d^\ast ) \dashv b
  \big),
  \\
  &
  \big( \, (a,b), \, (c,d), \, (e,f) \, \big)_\vdash
  =
  \\
  &\quad
  \big(
  ( a \vdash c ) \vdash e
  - ( d \dashv b^\ast ) \vdash e 
  - f \dashv ( d^\ast \dashv a ) 
  - f \dashv ( b^\ast \vdash c^\ast )
  \\
  &\qquad
  - a \vdash ( c \vdash e )
  + a \vdash ( f \dashv d^\ast )
  + ( c^\ast \vdash f ) \dashv b^\ast
  + ( e \dashv d ) \dashv b^\ast,
  \\
  &\qquad
  ( c^\ast \dashv a^\ast ) \vdash f
  - ( b \vdash d^\ast ) \vdash f
  + e \dashv ( a^\ast \vdash d )
  + e \dashv ( c \dashv b )
  \\
  &\qquad
  - a^\ast \vdash ( c^\ast \vdash f )
  - a^\ast \vdash ( e \dashv d )
  - ( c \vdash e ) \dashv b
  + ( f \dashv d^\ast ) \dashv b
  \big).
  \end{align*}
\end{lemma}

\begin{lemma} \label{altlemma1}
Let $A$ be a 0-dialgebra with involution.  
In the Cayley-Dickson double $A \oplus A$, 
the result of evaluating the first identity in Definition \ref{alternativedialgebra} is
  \begin{align*}
  &
  \big( \, (a,b), \, (c,d), \, (e,f) \, \big)_\dashv + \big( \, (e,f), \, (c,d), 
  \, (a,b) \, \big)_\vdash
  =
  \\
  &\quad
  \big(
  ( a \dashv c ) \dashv e
  - a \dashv ( c \dashv e )
  + ( e \vdash c ) \vdash a  
  - e \vdash ( c \vdash a ) 
  \\
  &\qquad
  + ( a \dashv d ) \dashv f^\ast
  - ( d \dashv f^\ast ) \vdash a     
  + a \dashv ( f \vdash d^\ast ) 
  - f \vdash ( d^\ast \vdash a ) 
  \\
  &\qquad \quad
  + ( c^\ast \vdash b ) \dashv f^\ast 
  - b \dashv ( f^\ast \vdash c^\ast ) 
  + ( c^\ast \dashv f ) \vdash b^\ast 
  - f \vdash ( b^\ast \dashv c^\ast ) 
  \\
  &\qquad \quad \quad
  - b \dashv ( d^\ast \dashv e ) 
  + e \vdash ( b \dashv d^\ast ) 
  + ( e \vdash d ) \vdash b^\ast 
  - ( d \vdash b^\ast ) \dashv e,
  \\
  &\qquad
  - a^\ast \dashv ( c^\ast \dashv f )  
  + ( c^\ast \vdash a^\ast ) \dashv f 
  + a \dashv ( c \dashv f )
  - ( c \vdash a ) \dashv f 
  \\
  &\qquad \quad
  - a^\ast \dashv ( e \vdash d ) 
  + a \dashv ( e^\ast \vdash d ) 
  + e \vdash ( a^\ast \dashv d ) 
  - e^\ast \vdash ( a \dashv d ) 
  \\
  &\qquad \quad \quad  
  + e \vdash ( c \vdash b ) 
  - ( c \dashv e ) \vdash b 
  + ( c^\ast \dashv e^\ast ) \vdash b 
  - e^\ast \vdash ( c^\ast \vdash b ) 
  \\
  &\qquad \quad \quad \quad  
  - ( b \dashv d^\ast ) \dashv f
  + ( f \vdash d^\ast ) \vdash b 
  - ( f \vdash d^\ast ) \vdash b 
  + ( b \dashv d^\ast ) \dashv f
  \big).
  \end{align*} 
\end{lemma}

\begin{lemma} \label{altlemma2}
Let $A$ be a 0-dialgebra with involution.  
In the Cayley-Dickson double $A \oplus A$, 
the result of evaluating the second identity in Definition \ref{alternativedialgebra} is
  \begin{align*}
  &
  \big( \, (a,b), \, (c,d), \, (e,f) \, \big)_\dashv - \big( \, (c,d), \, (e,f), 
  \, (a,b) \, \big)_\vdash
  \\
  &=
  \big(
  ( a \dashv c ) \dashv e
  - a \dashv ( c \dashv e ) 
  - ( c \vdash e ) \vdash a
  + c \vdash ( e \vdash a ) 
  \\
  &\qquad
  + a \dashv ( f \vdash d^\ast ) 
  - f \vdash ( d^\ast \vdash a )
  - ( a \dashv f ) \dashv d^\ast 
  + ( f \dashv d^\ast ) \vdash a 
  \\
  &\qquad \quad
  + b \dashv ( f^\ast \dashv c ) 
  + ( c^\ast \dashv f ) \vdash b^\ast 
  - f \vdash ( b^\ast \dashv c^\ast ) 
  - c \vdash ( b \dashv f^\ast ) 
  \\
  &\qquad \quad \quad
  + b \dashv ( d^\ast \vdash e^\ast ) 
  - ( e^\ast \vdash b ) \dashv d^\ast 
  - ( d \vdash b^\ast ) \dashv e
  + ( e \vdash d ) \vdash b^\ast,
  \\
  &\qquad
  - a \dashv ( c^\ast \vdash f ) 
  + c^\ast \vdash ( a \dashv f ) 
  - a^\ast \dashv ( c^\ast \dashv f ) 
  + ( c^\ast \vdash a^\ast ) \dashv f 
  \\
  &\qquad \quad
  - a \dashv ( e \dashv d ) 
  + ( e \vdash a ) \dashv d 
  + e \vdash ( a^\ast \dashv d ) 
  - a^\ast \dashv ( e \vdash d ) 
  \\
  &\qquad \quad \quad
  + e \vdash ( c \vdash b )
  - ( c \dashv e ) \vdash b 
  - ( e^\ast \dashv c^\ast ) \vdash b 
  + c^\ast \vdash ( e^\ast \vdash b ) 
  \\
  &\qquad \quad \quad \quad
  - ( b \dashv d^\ast ) \dashv f
  + ( f \vdash d^\ast ) \vdash b
  + ( d \vdash f^\ast ) \vdash b 
  - ( b \dashv f^\ast ) \dashv d 
  \big).
  \end{align*} 
\end{lemma}  
  
\begin{lemma} \label{altlemma3}
Let $A$ be a 0-dialgebra with involution.  
In the Cayley-Dickson double $A \oplus A$, 
the result of evaluating the third identity in Definition \ref{alternativedialgebra} is
  \begin{align*}
  &
  \big( \, (a,b), \, (c,d), \, (e,f) \, \big)_\times + \big( \, (a,b), \, (e,f), 
  \, (c,d) \, \big)_\vdash
  \\
  &=
  \big(
  ( a \vdash c ) \dashv e
  - a \vdash ( c \dashv e ) 
  + ( a \vdash e ) \vdash c 
  - a \vdash ( e \vdash c )  
  \\
  &\qquad
  + a \vdash ( f \vdash d^\ast ) 
  + a \vdash ( d \dashv f^\ast ) 
  - d \dashv ( f^\ast \dashv a ) 
  - f \vdash ( d^\ast \dashv a ) 
  \\
  &\qquad \quad
  + ( c \dashv f ) \dashv b^\ast 
  - ( f \dashv b^\ast ) \vdash c 
  + ( c^\ast \dashv f ) \dashv b^\ast 
  - f \vdash ( b^\ast \vdash c^\ast )
  \\
  &\qquad \quad \quad
  + ( e \vdash d ) \dashv b^\ast
  - ( d \dashv b^\ast ) \dashv e 
  - d \dashv ( b^\ast \vdash e^\ast ) 
  + ( e^\ast \vdash d ) \dashv b^\ast,
  \\
  &\qquad
  - a^\ast \vdash ( c \dashv f )
  + c \dashv ( a^\ast \vdash f ) 
  + ( c^\ast \dashv a^\ast ) \dashv f 
  - a^\ast \vdash ( c^\ast \dashv f ) 
  \\
  &\qquad \quad
  + e \vdash ( a^\ast \vdash d )
  - a^\ast \vdash ( e \vdash d ) 
  - a^\ast \vdash ( e^\ast \vdash d ) 
  + ( e^\ast \dashv a^\ast ) \vdash d 
  \\
  &\qquad \quad \quad
  + e \vdash ( c \dashv b )
  + c \dashv ( e \dashv b ) 
  - ( e \vdash c ) \dashv b 
  - ( c \dashv e ) \dashv b 
  \\
  &\qquad \quad \quad \quad
  - ( b \vdash d^\ast ) \dashv f 
  + ( f \vdash d^\ast ) \dashv b 
  - ( b \vdash f^\ast ) \vdash d 
  + ( d \dashv f^\ast ) \dashv b 
  \big).
  \end{align*} 
\end{lemma}    
  
\begin{theorem} \label{alter}
If $A$ is a partially symmetric associative 0-dialgebra with involution, 
then $A \oplus A$ is an partially symmetric alternative 0-dialgebra with involution.
\end{theorem}
 
\begin{proof}
It remains to verify that the expressions in Lemmas \ref{altlemma1} to \ref{altlemma3} 
vanish in every partially symmetric associative 0-dialgebra.
We collect groups of four terms involving the same variables;
in every case, the result is 0.
For the first identity:
  \allowdisplaybreaks
  \begin{align*}
  ( a \dashv c ) \dashv e
  - a \dashv ( c \dashv e )
  + ( e \vdash c ) \vdash a  
  - e \vdash ( c \vdash a ) 
  &=
  ( a, c, e )_\dashv + ( e, c, a )_\vdash,
  \\
  ( a \dashv d ) \dashv f^\ast
  - ( d \dashv f^\ast ) \vdash a     
  + a \dashv ( f \vdash d^\ast ) 
  - f \vdash ( d^\ast \vdash a )
  &= 
  \{ a, \mathrm{sym}( d \dashv f^\ast ) \},
  \\
  ( c^\ast \vdash b ) \dashv f^\ast 
  - b \dashv ( f^\ast \vdash c^\ast ) 
  + ( c^\ast \dashv f ) \vdash b^\ast 
  - f \vdash ( b^\ast \dashv c^\ast ) 
  &=
  - \{ \mathrm{sym}( b \dashv f^\ast ), c^\ast \},
  \\
  - b \dashv ( d^\ast \dashv e ) 
  + e \vdash ( b \dashv d^\ast ) 
  + ( e \vdash d ) \vdash b^\ast 
  - ( d \vdash b^\ast ) \dashv e
  &=
  \{ \mathrm{sym}( b \dashv d^\ast ), e \},
  \\
  - a^\ast \dashv ( c^\ast \dashv f )  
  + ( c^\ast \vdash a^\ast ) \dashv f 
  + a \dashv ( c \dashv f )
  - ( c \vdash a ) \dashv f 
  &=
  \mathrm{sym}( \{ a, c \} ) \dashv f,
  \\
  - a^\ast \dashv ( e \vdash d ) 
  + a \dashv ( e^\ast \vdash d ) 
  + e \vdash ( a^\ast \dashv d ) 
  - e^\ast \vdash ( a \dashv d )
  &= 
  \mathrm{sym}( \{ a, e^\ast \} ) \dashv d,
  \\
  e \vdash ( c \vdash b ) 
  - ( c \dashv e ) \vdash b 
  + ( c^\ast \dashv e^\ast ) \vdash b 
  - e^\ast \vdash ( c^\ast \vdash b ) 
  &=
  \mathrm{sym}( \{ c^\ast, e^\ast \} ) \vdash b,
  \\
  - ( b \dashv d^\ast ) \dashv f
  + ( f \vdash d^\ast ) \vdash b 
  - ( f \vdash d^\ast ) \vdash b 
  + ( b \dashv d^\ast ) \dashv f
  &=
  0.
  \end{align*} 
For the second identity:
  \allowdisplaybreaks
  \begin{align*}
  ( a \dashv c ) \dashv e
  - a \dashv ( c \dashv e ) 
  - ( c \vdash e ) \vdash a
  + c \vdash ( e \vdash a )
  &= 
  ( a, c, e )_\dashv - ( c, e, a )_\vdash,
  \\
  a \dashv ( f \vdash d^\ast ) 
  - f \vdash ( d^\ast \vdash a )
  - ( a \dashv f ) \dashv d^\ast 
  + ( f \dashv d^\ast ) \vdash a 
  &= 
  0,
  \\
  b \dashv ( f^\ast \dashv c ) 
  + ( c^\ast \dashv f ) \vdash b^\ast 
  - f \vdash ( b^\ast \dashv c^\ast ) 
  - c \vdash ( b \dashv f^\ast ) 
  &= 
  \mathrm{sym}( b \dashv f^\ast, c ),
  \\
  b \dashv ( d^\ast \vdash e^\ast ) 
  - ( e^\ast \vdash b ) \dashv d^\ast 
  - ( d \vdash b^\ast ) \dashv e
  + ( e \vdash d ) \vdash b^\ast
  &= 
  - \mathrm{sym}( d \vdash b^\ast, e ),
  \\
  - a \dashv ( c^\ast \vdash f ) 
  + c^\ast \vdash ( a \dashv f ) 
  - a^\ast \dashv ( c^\ast \dashv f ) 
  + ( c^\ast \vdash a^\ast ) \dashv f 
  &= 
  - \{ \mathrm{sym}(a), c \} \dashv f,
  \\
  - a \dashv ( e \dashv d ) 
  + ( e \vdash a ) \dashv d 
  + e \vdash ( a^\ast \dashv d ) 
  - a^\ast \dashv ( e \vdash d ) 
  &= 
  - \{ \mathrm{sym}(a), e \} \dashv d,
  \\
  e \vdash ( c \vdash b )
  - ( c \dashv e ) \vdash b 
  - ( e^\ast \dashv c^\ast ) \vdash b 
  + c^\ast \vdash ( e^\ast \vdash b ) 
  &= 
  \mathrm{sym}( \{ e, c \} ) \vdash b,
  \\
  - ( b \dashv d^\ast ) \dashv f
  + ( f \vdash d^\ast ) \vdash b
  + ( d \vdash f^\ast ) \vdash b 
  - ( b \dashv f^\ast ) \dashv d 
  &= 
  -\{ b, \mathrm{sym}( d^\ast \dashv f ) \};
  \end{align*} 
in the last case we have used Lemma \ref{movestar}.  
For the third identity:
  \allowdisplaybreaks
  \begin{align*}
  ( a \vdash c ) \dashv e
  - a \vdash ( c \dashv e ) 
  + ( a \vdash e ) \vdash c 
  - a \vdash ( e \vdash c )
  &=  
  ( a, c, e )_\times + ( a, e, c )_\vdash,
  \\
  a \vdash ( f \vdash d^\ast ) 
  + a \vdash ( d \dashv f^\ast ) 
  - d \dashv ( f^\ast \dashv a ) 
  - f \vdash ( d^\ast \dashv a ) 
  &=  
  - \{ \mathrm{sym}( d \dashv f^\ast ), a \},
  \\
  ( c \dashv f ) \dashv b^\ast 
  - ( f \dashv b^\ast ) \vdash c 
  + ( c^\ast \dashv f ) \dashv b^\ast 
  - f \vdash ( b^\ast \vdash c^\ast )
  &=  
  \{ \mathrm{sym}(c), f \dashv b^\ast \},
  \\
  ( e \vdash d ) \dashv b^\ast
  - ( d \dashv b^\ast ) \dashv e 
  - d \dashv ( b^\ast \vdash e^\ast ) 
  + ( e^\ast \vdash d ) \dashv b^\ast
  &=  
  - \{ d \dashv b^\ast, \mathrm{sym}(e) \},
  \\
  - a^\ast \vdash ( c \dashv f )
  + c \dashv ( a^\ast \vdash f ) 
  + ( c^\ast \dashv a^\ast ) \dashv f 
  - a^\ast \vdash ( c^\ast \dashv f ) 
  &=  
  \{ \mathrm{sym}(c), a^\ast \} \dashv f,
  \\
  e \vdash ( a^\ast \vdash d )
  - a^\ast \vdash ( e \vdash d ) 
  - a^\ast \vdash ( e^\ast \vdash d ) 
  + ( e^\ast \dashv a^\ast ) \vdash d 
  &=  
  \{ \mathrm{sym}(e), a^\ast \} \vdash d,
  \\
  - ( e \vdash c ) \dashv b 
  + e \vdash ( c \dashv b )
  - ( c \dashv e ) \dashv b 
  + c \dashv ( e \dashv b ) 
  &=  
  - ( e, c, b )_\times - ( c, e, b )_\dashv,
  \\
  - ( b \vdash d^\ast ) \dashv f 
  + ( f \vdash d^\ast ) \dashv b 
  - ( b \vdash f^\ast ) \vdash d 
  + ( d \dashv f^\ast ) \dashv b 
  &=  
  \{ \mathrm{sym}( d \dashv f^\ast ), b \};
  \end{align*}
here again we have used Lemma \ref{movestar}.
\end{proof}


\section{Alternative dialgebras and Jordan dialgebras}

In this section we make a slight digression and consider the polynomial identities satisfied by
bilinear operations in associative and alternative dialgebras.

The Leibniz bracket in an associative 0-dialgebra satisfies the defining identities for Leibniz algebras.
In an alternative 0-dialgebra, the Leibniz bracket satisfies the defining identities for Malcev dialgebras \cite{BPSO}.

The Jordan diproduct $ab = a \dashv b + b \vdash a$ in an associative 0-dialgebra satisfies the defining identities
for Jordan dialgebras. 

\begin{definition} \label{JDdefinition}
\cite{Bremner,Kolesnikov,Felipe}
Over a field of characteristic not 2 or 3, a \textbf{(right) Jordan dialgebra} 
is a vector space with a bilinear operation $ab$ satisfying
  \begin{alignat*}{2}
  &\textbf{right commutativity:}
  &\qquad
  a(bc) &\equiv a(cb),
  \\
  &\textbf{right Jordan identity:}
  &\qquad
  ( b a^2 ) a &\equiv ( b a ) a^2,
  \\
  &\textbf{right Osborn identity:}
  &\qquad
  ( a, b, c^2 ) &\equiv 2 ( a c, b, c ).
  \end{alignat*}
\end{definition}

For the Jordan diproduct in an alternative 0-dialgebra, we have the following theorem, originally proved by 
Gubarev and Kolesnikov \cite[Example 2, p.~505]{GK} using the theory of conformal algebras.  

\begin{theorem}
If $A$ is an alternative 0-dialgebra, then the underlying vector space of $A$ becomes 
a Jordan dialgebra with respect to the Jordan diproduct.
\end{theorem}

\begin{proof}
Right commutativity follows easily from the bar identities.
Next, we check the right Jordan identity. 
For $(ba)a^2$ we obtain:
\allowdisplaybreaks
\begin{align*}
&
( b \dashv a + a \vdash b ) \dashv ( a \dashv a + a \vdash a ) 
+ 
( a \dashv a + a \vdash a ) \vdash ( b \dashv a + a \vdash b ) 
\\
&= 
( b \dashv a ) \dashv ( a \dashv a ) + ( a \vdash b ) \dashv ( a \dashv a ) 
+ 
( b \dashv a ) \dashv ( a \vdash a ) + ( a \vdash b ) \dashv ( a \vdash a )
\\
&\quad
+ 
( a \dashv a ) \vdash ( b \dashv a ) + ( a \vdash a ) \vdash ( b \dashv a ) 
+ 
( a \dashv a ) \vdash ( a \vdash b ) + ( a \vdash a ) \vdash ( a \vdash b )  
\\
& \equiv 
2 ( b \dashv a ) \dashv ( a \dashv a ) + 2 ( a \vdash b ) \dashv ( a \dashv a ) 
+ 
2 ( a \vdash a ) \vdash ( b \dashv a ) + 2 ( a \vdash a ) \vdash (a \vdash b ).
\end{align*}
On the other hand, for $(b a^2 ) a$ we obtain:
\begin{align*}
&
\big( b \dashv ( a \dashv a + a \vdash a ) + ( a \dashv a + a \vdash a ) \vdash b \big) \dashv a 
\\
&\quad 
+  
a \vdash \big( b \dashv ( a \dashv a + a \vdash a ) + ( a \dashv a + a \vdash a ) \vdash b \big) 
\\
&= 
\big( b \dashv ( a \dashv a ) \big) \dashv a 
+ 
\big( b \dashv ( a \vdash a ) \big) \dashv a 
+ 
\big( ( a \dashv a ) \vdash b \big) \dashv a 
+ 
\big( ( a \vdash a ) \vdash b \big) \dashv a 
\\
&\quad
+  
a \vdash \big( b \dashv ( a \dashv a ) \big) 
+ 
a \vdash \big( b \dashv ( a \vdash a ) \big) 
+ 
a \vdash \big( ( a \dashv a ) \vdash b \big)
+ 
a \vdash \big( ( a \vdash a ) \vdash b \big) 
\\
&\equiv
2 \big( b \dashv ( a \dashv a ) \big) \dashv a
+ 
2 \big( ( a \vdash a ) \vdash b \big) \dashv a
+
2a \vdash \big( b \dashv ( a \dashv a ) \big) 
+ 
2a \vdash \big( ( a \vdash a ) \vdash b \big).
\end{align*}
Since $A$ is an alternative 0-dialgebra \cite{Liu} we have the identities,
\[
( x, y, y )_\dashv \equiv 0, 
\qquad 
( x, y, x )_\times \equiv 0,
\qquad
( x, x, y )_\vdash \equiv 0,
\]
or more explicitly,
\[
( x \dashv y ) \dashv y \equiv x \dashv ( y \dashv y), 
\quad
( x \vdash y ) \dashv x \equiv x \vdash ( y \dashv x),
\quad 
( x \vdash x ) \vdash y \equiv x \vdash ( x \vdash y).
\quad 
\]
Applying these identities, we see that $(ba^2)a$ equals
\begin{align*}
&
2 \big( ( b \dashv a ) \dashv a \big) \dashv a
+ 
2 \big( a \vdash ( a \vdash b ) \big) \dashv a 
+
2a \vdash \big( ( b \dashv a ) \dashv a \big) 
+ 
2a \vdash \big( a \vdash ( a \vdash b ) \big) 
\\
&\equiv 
2 \big( ( b \dashv a ) \dashv a \big) \dashv a 
+
2 a \vdash \big( ( a \vdash b ) \dashv a \big)
+
2 \big( a \vdash ( b \dashv a ) \big) \dashv a 
+ 
2 a \vdash \big( a \vdash ( a \vdash b ) \big)
\\
&\equiv 
2 \big( ( b \dashv a ) \dashv a \big) \dashv a 
+
2 a \vdash \big( a \vdash ( b \dashv a ) \big)
+
2 \big( ( a \vdash b ) \dashv a \big) \dashv a 
+ 
2 a \vdash \big( a \vdash ( a \vdash b ) \big)
\end{align*}
and this is equivalent to the expression for $( ba ) a^2$.

We used the computer algebra system Maple to express the Osborn identity in terms of
the consequences in degree 4 of the defining identities for alternative 0-dialgebras.
The simplest expression we found had 61 terms and coefficients $\pm 1, \pm 2$.
Since a formula of such complexity does not have much intrinsic interest, and will not
be needed in the rest of this paper, we decided not to include it.
\end{proof}


\section{From alternative dialgebras to flexible dialgebras}

Applying the KP algorithm to the flexible identity,
  \[
  (ab)c - a(bc) + (cb)a - c(ba) \equiv 0,
  \]
gives the bar identities together with these three identities,
  \allowdisplaybreaks
  \begin{align*}
  ( a \dashv b ) \dashv c - a \dashv ( b \dashv c ) + ( c \vdash b ) \vdash a - c \vdash ( b \vdash a ) &\equiv 0,
  \\
  ( a \vdash b ) \dashv c - a \vdash ( b \dashv c ) + ( c \vdash b ) \dashv a - c \vdash ( b \dashv a ) &\equiv 0,
  \\  
  ( a \vdash b ) \vdash c - a \vdash ( b \vdash c ) + ( c \dashv b ) \dashv a - c \dashv ( b \dashv a ) &\equiv 0,
  \end{align*}  
where the first and third are equivalent.

\begin{definition} \label{flexibledialgebra}
A \textbf{flexible 0-dialgebra} satisfies the bar identities together with
  \[
  (a,b,c)_\dashv + (c,b,a)_\vdash \equiv 0,
  \qquad
  (a,b,c)_\times + (c,b,a)_\times \equiv 0.
  \]
The first identity coincides with the first identity in the definition of alternative 0-dialgebra.
Every alternative 0-dialgebra is a flexible 0-dialgebra.
\end{definition}

\begin{lemma} \label{flexlemma1}
Let $A$ be a 0-dialgebra with involution. 
In the Cayley-Dickson double $A \oplus A$, the result of evaluating 
the second identity in Definition \ref{flexibledialgebra} is as follows:
  \begin{align*}
  &
  \big( \, (a,b), \, (c,d), \, (e,f) \, \big)_\times + \big( \, (e,f), \, (c,d), 
  \, (a,b) \, \big)_\times
  \\
  &=
  \big(
  ( a \vdash c ) \dashv e
  - a \vdash ( c \dashv e ) 
  + ( e \vdash c ) \dashv a 
  - e \vdash ( c \dashv a )  
  \\
  &\qquad
  - f \vdash ( d^\ast \dashv a ) 
  + a \vdash ( f \vdash d^\ast )   
  - (d \dashv  f^\ast) \dashv a  
  + (a \vdash  d) \dashv f^\ast
  \\
  &\qquad \quad
  -  f \vdash (b^\ast \vdash c^\ast) 
  + ( c^\ast \dashv f ) \dashv b^\ast 
  - b \vdash ( f^\ast \vdash c^\ast )
  + ( c^\ast \dashv b ) \dashv f^\ast 
  \\
  &\qquad \quad \quad
  - ( d \dashv b^\ast ) \dashv e
  + ( e \vdash d ) \dashv b^\ast 
  - b \vdash ( d^\ast \dashv e^\ast ) 
  + e \vdash (b \vdash d^\ast),
  \\
  &\qquad
  ( c^\ast \dashv a^\ast ) \dashv f 
  - a^\ast \vdash ( c^\ast \dashv f )
  + a \vdash ( c \dashv f )
  - (c \dashv a) \dashv f   
  \\
  &\qquad \quad
  + e \vdash ( a^\ast \vdash d )
  - a^\ast \vdash ( e \vdash d ) 
  + a \vdash ( e^\ast \vdash d ) 
  - e^\ast \vdash ( a \vdash d )
  \\
  &\qquad \quad \quad
  + e \vdash ( c \dashv b )
  - ( c \dashv e ) \dashv b 
  + ( c^\ast \dashv e^\ast ) \dashv b  
  - e^\ast \vdash ( c^\ast \dashv b ) 
  \\
  &\qquad \quad \quad \quad
  - ( b \vdash d^\ast ) \dashv f 
  + ( f \vdash d^\ast ) \dashv b 
  - ( f \vdash d^\ast ) \dashv b 
  + ( b \vdash d^\ast ) \dashv f 
  \big).
  \end{align*} 
\end{lemma} 

\begin{proof}
Straightforward calculation using the middle equation of Lemma \ref{oldLemma5.8}.
\end{proof}

To prove the next theorem, we need to impose further conditions on a partially symmetric 0-dialgebra with involution.
These conditions are the KP identities corresponding to the condition that symmetric elements 
in an algebra with involution associate with every element:
\[
( ( x {+} x^\ast ) y )z \equiv ( x {+} x^\ast)( yz ),
\;\;
( yz )( x {+} x^\ast ) \equiv y( z ( x {+} x^\ast ) ),
\;\;
( y ( x {+} x^\ast ) )z \equiv y( ( x {+} x^\ast ) z ).
\]
We make $x$, $y$ and $z$ in turn the central argument, 
write the results in terms of associators,
and obtain the following definition.

\begin{definition}
A partially symmetric 0-dialgebra is a \textbf{symmetric 0-dialgebra} if 
the following expressions vanish identically:
\begin{equation}
\label{Juana}
\left\{
\begin{array}{lll}
( x, y, z )_\dashv + ( x^\ast, y, z )_\dashv,
&
( x, y, z )_\times + ( x^\ast, y, z )_\times,
&
( x, y, z )_\vdash + ( x^\ast, y, z )_\vdash,
\\[4pt]
( x, y, z )_\dashv + ( x, y^\ast, z )_\dashv,
&
( x, y, z )_\times + ( x, y^\ast, z )_\times,
&
( x, y, z )_\vdash + ( x, y^\ast, z )_\vdash,
\\[4pt]
( x, y, z )_\dashv + ( x, y, z^\ast )_\dashv,
&
( x, y, z )_\times + ( x, y, z^\ast )_\times,
&
( x, y, z )_\vdash + ( x, y, z^\ast )_\vdash.
\end{array}
\right.
\end{equation}
\end{definition}

\begin{lemma} \label{lemmaflexible}
The following expressions vanish identically in every flexible 0-dialgebra with involution
which satisfies the equations \eqref{Juana}: 
\[
\begin{array}{lll}
( x, y, z )_\dashv + ( z^\ast, y^\ast, x )_\vdash,
&
( x, y, z )_\dashv + ( z, y^\ast, x^\ast )_\vdash,
&
( x, y, z )_\times + ( z^\ast, y^\ast, x )_\times,
\\[4pt]
( x, y, z )_\times + ( z, y^\ast, x^\ast )_\times,
&
( x, y, z )_\dashv - ( x^\ast, y^\ast, z )_\dashv,
&
( x, y, z )_\times - ( x^\ast, y^\ast, z )_\times.
\end{array}
\]
\end{lemma} 

\begin{proof}
Use equations \eqref{Juana} to add stars to the variables, 
and the flexible 0-dialgebra identities to change from one type of associator to another.
\end{proof} 

\begin{theorem}
If $A$ is a flexible symmetric 0-dialgebra, 
then its Cayley-Dickson double $A \oplus A$ is also a flexible symmetric 0-dialgebra.
In particular, this holds if $A$ is an alternative symmetric 0-dialgebra.
\end{theorem}
 
\begin{proof}
It suffices to check that the expressions in Lemmas \ref{altlemma1} and \ref{flexlemma1} vanish
in every flexible symmetric 0-dialgebra.
As in the proof of Theorem \ref{alter}, we consider separately the groups of four terms involving 
the same variables. In every case, the result is 0 by the flexibility and symmetry of $A$, 
and applications of Lemmas \ref{movestar} and \ref{lemmaflexible}.
For the first identity we have:
  \allowdisplaybreaks
  \begin{align*}
  &
  ( a \dashv c ) \dashv e
  - a \dashv ( c \dashv e )
  + ( e \vdash c ) \vdash a  
  - e \vdash ( c \vdash a ) 
  =
  ( a, c, e )_\dashv 
  + ( e, c, a )_\vdash,
  \\
  &
  ( a \dashv d ) \dashv f^\ast
  - ( d \dashv f^\ast ) \vdash a     
  + a \dashv ( f \vdash d^\ast ) 
  - f \vdash ( d^\ast \vdash a )
  = 
  \\
  &\qquad
  \{ a, \mathrm{sym}( d \dashv f^\ast ) \}
  + ( f, d^\ast, a )_\vdash 
  + ( a, d, f^\ast )_\dashv,
  \\
  &
  ( c^\ast \vdash b ) \dashv f^\ast 
  - b \dashv ( f^\ast \vdash c^\ast ) 
  + ( c^\ast \dashv f ) \vdash b^\ast 
  - f \vdash ( b^\ast \dashv c^\ast ) 
  =
  \\
  &\qquad
  - \{ \mathrm{sym}( b \dashv f^\ast ), c^\ast \}
  + ( c^\ast, f, b^\ast )_\vdash + ( b, f^\ast, c^\ast )_\dashv 
  + ( f, b^\ast, c^\ast )_\times + ( c^\ast, b, f^\ast)_\times,
  \\
  &
  - b \dashv ( d^\ast \dashv e ) 
  + e \vdash ( b \dashv d^\ast ) 
  + ( e \vdash d ) \vdash b^\ast 
  - ( d \vdash b^\ast ) \dashv e
  =
  \\
  &\qquad
  - \{ \mathrm{sym}( b \dashv d^\ast ), e \}
  + ( e, d, b^\ast )_\vdash + ( b, d^\ast, e )_\dashv,
  \\
  &
  - a^\ast \dashv ( c^\ast \dashv f )  
  + ( c^\ast \vdash a^\ast ) \dashv f 
  + a \dashv ( c \dashv f )
  - ( c \vdash a ) \dashv f 
  =
  \\
  &\qquad
  \big( \mathrm{sym}( a \dashv c ) - \mathrm{sym}( a^\ast \dashv c^\ast ) \big) \dashv f
  + ( a^\ast, c^\ast, f)_\dashv - ( a, c, f)_\dashv,
  \\
  &
  - a^\ast \dashv ( e \vdash d ) 
  + a \dashv ( e^\ast \vdash d ) 
  + e \vdash ( a^\ast \dashv d ) 
  - e^\ast \vdash ( a \dashv d )
  = 
  \\
  &\qquad
  \big( \mathrm{sym}( a \dashv e^\ast ) 
  - \mathrm{sym}( a^\ast \dashv e ) \big) \dashv d
  - ( e, a^\ast, d )_\times 
  + ( e^\ast, a, d )_\times,
  \\
  &
  e \vdash ( c \vdash b ) 
  - ( c \dashv e ) \vdash b 
  + ( c^\ast \dashv e^\ast ) \vdash b 
  - e^\ast \vdash ( c^\ast \vdash b ) 
  =
  \\
  &\qquad
  \big( \mathrm{sym}( e \dashv c ) 
  - \mathrm{sym}( e^\ast \dashv c^\ast ) \big) \vdash b 
  + ( e^\ast, c^\ast, b )_\vdash - ( e, c, b )_\vdash,  
  \\
  &
  - ( b \dashv d^\ast ) \dashv f
  + ( f \vdash d^\ast ) \vdash b 
  - ( f \vdash d^\ast ) \vdash b 
  + ( b \dashv d^\ast ) \dashv f
  =
  0.
  \end{align*} 
For the second identity we have:
  \allowdisplaybreaks
  \begin{align*}
  &
  ( a \vdash c ) \dashv e
  - a \vdash ( c \dashv e ) 
  + ( e \vdash c ) \dashv a 
  - e \vdash ( c \dashv a ) 
  = 
  ( a, c, e )_\times 
  + ( e, c, a )_\times,
  \\
  &
  - f \vdash ( d^\ast \dashv a ) 
  + a \vdash ( f \vdash d^\ast )   
  - ( d \dashv  f^\ast ) \dashv a  
  + ( a \vdash  d ) \dashv f^\ast
  = 
  \\
  &\qquad
  - \{\mathrm{sym}( d \dashv f^\ast ), a \}  
  + ( f, d^\ast, a )_\times 
  + ( a, d, f^\ast )_\times,
  \\
  &
  -  f \vdash ( b^\ast \vdash c^\ast ) 
  + ( c^\ast \dashv f ) \dashv b^\ast 
  - b \vdash ( f^\ast \vdash c^\ast )
  + ( c^\ast \dashv b ) \dashv f^\ast 
  =
  \\
  &\qquad
  \{ c^\ast,  \mathrm{sym}( b \dashv f^\ast ) \}
  + ( f, b^\ast, c^\ast )_\vdash 
  + ( c^\ast, f, b^\ast)_\dashv 
  + ( c^\ast, b, f^\ast )_
  \dashv + ( b, f^\ast, c^\ast )_\vdash,
  \\
  &
  - ( d \dashv b^\ast ) \dashv e
  + ( e \vdash d ) \dashv b^\ast 
  - b \vdash ( d^\ast \dashv e^\ast ) 
  + e \vdash ( b \vdash d^\ast )
  =
  \\
  &\qquad
  - \{ \mathrm{sym}( b \dashv d^\ast ), e \}
  + ( e, d, b^\ast )_\times 
  + ( b, d^\ast, e )_\times,
  \\
  &
  ( c^\ast \dashv a^\ast ) \dashv f 
  - a^\ast \vdash ( c^\ast \dashv f )
  + a \vdash ( c \dashv f )
  - ( c \dashv a ) \dashv f
  =
  \\
  &\qquad
  \big( \mathrm{sym}( c^\ast \dashv a^\ast ) 
  - \mathrm{sym}( c \dashv a ) \big) \dashv 
  f + ( a^\ast, c^\ast, f )_\times 
  - ( a, c, f )_\times,
  \\
  &
  e \vdash ( a^\ast \vdash d )
  - a^\ast \vdash ( e \vdash d ) 
  + a \vdash ( e^\ast \vdash d ) 
  - e^\ast \vdash ( a \vdash d )
  = 
  \\
  &\qquad
  \big( \mathrm{sym}( e \dashv a^\ast ) 
  - \mathrm{sym}( e^\ast \dashv a ) \big) \vdash d
  + ( a^\ast, e, d )_\vdash 
  - ( e, a^\ast, d )_\vdash 
  \\
  &\qquad
  - ( a, e^\ast, d)_\vdash 
  + (e^\ast, a, d)_\vdash,
  \\
  &
  e \vdash ( c \dashv b )
  - ( c \dashv e ) \dashv b 
  + ( c^\ast \dashv e^\ast ) \dashv b  
  - e^\ast \vdash ( c^\ast \dashv b )
  =
  \\
  &\qquad
  \big( \mathrm{sym}( c^\ast \dashv e^\ast) 
  - \mathrm{sym}( c \dashv e) \big) \dashv b
  - ( e, c, b )_\times 
  + ( e^\ast, c^\ast, b )_\times,
  \\
  &
  - ( b \vdash d^\ast ) \dashv f 
  + ( f \vdash d^\ast ) \dashv b 
  - ( f \vdash d^\ast ) \dashv b 
  + ( b \vdash d^\ast ) \dashv f
  = 0.
  \end{align*} 
This completes the proof.
\end{proof}


\section{Two-dimensional dialgebras with involution}

In order to find interesting examples of the constructions described in 
the previous sections, we must first determine the most appropriate example 
of a 2-dimensional dialgebra to which we may apply the doubling process 
to obtain dialgebras generalizing the quaternions and octonions.
A classification of 2-dimensional associative dialgebras has been given 
recently by Mart\'in \cite{Martin}, but our assumptions are not the same, 
so we take a different approach.

We begin with a vector space $D$ with basis $\{ x, y \}$ over a field $\mathbb{F}$ which 
has two bilinear operations, denoted $\dashv$ and $\vdash$, and called the left and right 
products.
We need to determine the constraints on the structure constants so that   
  \begin{itemize}
  \item
$D$ is an associative dialgebra which is proper in the sense that the left 
and right products do not coincide.
  \item
$D$ is commutative: $a \dashv b \equiv b \vdash a$.
  \item
$D$ has an involution $a \mapsto a^\ast$.
  \end{itemize}
The structure constants $a,b,c,d,e,f,g,h \in \mathbb{F}$ for the left product are
  \[
  x \dashv x = ax + by, \quad
  x \dashv y = cx + dy, \quad
  y \dashv x = ex + fy, \quad
  y \dashv y = gx + hy.
  \]
Dicommutativity implies that the structure constants for the right product are
  \[
  x \vdash x = ax + by, \quad
  x \vdash y = ex + fy, \quad
  y \vdash x = cx + dy, \quad
  y \vdash y = gx + hy.
  \]
Without loss of generality we assume that $x^\ast = y$ and $y^\ast = x$, and hence
  \[
  ( ax + by )^\ast = bx + ay.
  \]
Imposing the left bar identities implies four equations:
  \allowdisplaybreaks
  \begin{alignat}{2}
  a c - a e + c d - c f &= 0, 
  &\qquad
  b c - b e + d^2  - d f &= 0,
  \label{first}
  \\
  c e + d g - e^2  - f g &= 0, 
  &\qquad
  c f + d h - e f - f h &= 0.
  \end{alignat}
By commutativity, imposing the right bar identities gives the same four equations.
Imposing left associativity implies 12 equations:
  \allowdisplaybreaks
  \begin{alignat}{2}
  &
  b (c - e) = 0, \qquad b (d - f) = 0, 
  &\qquad
  &
  b g - c d = 0, \qquad b g - e f = 0, 
  \\
  &
  g (c - e) = 0, \qquad g (d - f) = 0,
  \\
  &
  a c - a e - c f + d e = 0,
  &\qquad
  &
  a d - b c + b h - d^2 = 0,
  \\
  &
  a f - b e + b h - f^2 = 0,
  &\qquad
  &
  a g - c^2  + c h - d g = 0,
  \\
  &
  a g - e^2  + e h - f g = 0, 
  &\qquad
  &
  c f - d e + d h - f h = 0.
  \end{alignat}
By commutativity, imposing right associativity gives the same 12 equations. 
Imposing inner associativity gives three equations:
  \begin{equation}
  b g - d e = 0, \qquad
  a d - b c + b h - d f = 0, \qquad
  a g - c e + e h - f g = 0.
  \end{equation}
Imposing the involution identities gives four equations:
  \begin{equation}
  d - e = 0, \qquad
  c - f = 0, \qquad
  b - g = 0, \qquad
  a - h = 0.
  \label{last}
  \end{equation}
Solving the system of equations \eqref{first}--\eqref{last} produces three families of solutions:
  \allowdisplaybreaks
  \begin{alignat*}{9}
  &\bullet &\quad 
  a &= h,  &\;\; b &= g,  &\;\; c &= -g,  &\;\; d &= -g,  &\;\; 
  e &= -g,  &\;\; f &= -g,  &\;\; g &= \text{free},  &\;\; h &= \text{free}
  \\
  &\bullet &\quad 
  a &= h,  &\;\; b &= 0,  &\;\; c &=  h,  &\;\; d &=  0,  &\;\; 
  e &=  0,  &\;\; f &=  h,  &\;\; g &= 0,  &\;\;  h &= \text{free}
  \\
  &\bullet &\quad 
  a &= h,  &\;\; b &= h,  &\;\; c &=  h,  &\;\; d &=  h,  &\;\; 
  e &=  h,  &\;\; f &=  h,  &\;\; g &= h,  &\;\;  h &= \text{free}
  \end{alignat*}
The four fundamental solutions are as follows:
  \allowdisplaybreaks
  \begin{alignat*}{9}
  &(1) &\quad
  a &= 0, &\;\; b &= 1, &\;\; c &= -1, &\;\; d &= -1, &\;\; e &= -1, &\;\; f &= -1, &\;\; g &= 1, &\;\; h &= 0
  \\
  &(2) &\quad
  a &= 1, &\;\; b &= 0, &\;\; c &=  0, &\;\; d &=  0, &\;\; e &=  0, &\;\; f &=  0, &\;\; g &= 0, &\;\; h &= 1
  \\
  &(3) &\quad
  a &= 1, &\;\; b &= 0, &\;\; c &=  1, &\;\; d &=  0, &\;\; e &=  0, &\;\; f &=  1, &\;\; g &= 0, &\;\; h &= 1
  \\
  &(4) &\quad
  a &= 1, &\;\; b &= 1, &\;\; c &=  1, &\;\; d &=  1, &\;\; e &=  1, &\;\; f &=  1, &\;\; g &= 1, &\;\; h &= 1
  \end{alignat*}
Only the third satisfies the condition that the left and right products are distinct.

\begin{lemma} \label{2ddialgebra}
Up to a scalar multiple, there is a unique 2-dimensional 0-dialgebra which is
commutative and associative, has an involution, and has distinct operations;
its structure constants are as follows:
  \allowdisplaybreaks
  \begin{alignat*}{4}
  x \dashv x &= x, &\qquad
  x \dashv y &= x, &\qquad
  y \dashv x &= y, &\qquad
  y \dashv y &= y, \\
  x \vdash x &= x, &\qquad
  x \vdash y &= y, &\qquad
  y \vdash x &= x, &\qquad
  y \vdash y &= y.
  \end{alignat*}
\end{lemma}

\begin{remark}
This dialgebra is the case $A = \mathbb{F}$ and $n = 2$ of Loday \cite[page 13, Example 2.2($h$)]{Loday3}.
It can be characterized by the fact that both $x$ and $y$ are bar-units.
We can write its structure constants using multiplication tables:
  \[
  \begin{array}{l|rrrr}
  \dashv & x & y \\
  \midrule
  x & x & x \\
  y & y & y
  \end{array}
  \qquad\qquad
  \begin{array}{l|rrrr}
  \vdash & x & y \\
  \midrule
  x & x & y \\
  y & x & y
  \end{array}
  \]
If we introduce the new basis $p = \frac12(x+y)$, $q = \frac12(x-y)$ then we obtain
  \[
  \begin{array}{l|rrrr}
  \dashv & p & q \\
  \midrule
  p & p & 0 \\
  q & q & 0
  \end{array}
  \qquad\qquad
  \begin{array}{l|rrrr}
  \vdash & p & q \\
  \midrule
  p & p & q \\
  q & 0 & 0
  \end{array}
  \]
We then see that $p$ is a bar-unit and $q$ is a bar-zero.
\end{remark}

We apply the doubling process to the dialgebra $D$ of Lemma \ref{2ddialgebra}
and obtain a 4-dimensional dialgebra $E$ with basis $\{p,q,r,s\}$ and these left and right products:
  \[
  \begin{array}{l|rrrr}
  \dashv & p & q & r & s \\
  \midrule
  p & p &\quad p &\;  s &\;  s \\
  q & q &\quad q &\;  r &\;  r \\
  r & r &\quad r &\; -q &\; -q \\
  s & s &\quad s &\; -p &\; -p
  \end{array}
  \qquad\qquad
  \begin{array}{l|rrrr}
  \vdash & p & q & r & s \\
  \midrule
  p & p &\quad q &\;  r &\;  s \\
  q & p &\quad q &\;  r &\;  s \\
  r & r &\quad s &\; -p &\; -q \\
  s & r &\quad s &\; -p &\; -q
  \end{array}
  \]
This is an associative dialgebra with involution defined by the equations
  \[
  p^\ast = q, \qquad
  q^\ast = p, \qquad
  r^\ast = -r, \qquad
  s^\ast = -s.
  \]
This dialgebra is not commutative; the corresponding Leibniz algebra has the following structure constants:
  \[
  \begin{array}{l|rrrr}
  [-,-] & p & q & r & s \\
  \midrule
  p & 0 &\quad 0 &\; -r+s &\; -r+s \\
  q & 0 &\quad 0 &\;  r-s &\;  r-s \\
  r & 0 &\quad 0 &\;  p-q &\;  p-q \\
  s & 0 &\quad 0 &\; -p+q &\; -p+q
  \end{array}
  \]
Applying the doubling process again produces an 8-dimensional dialgebra $F$ 
with basis $\{p,q,r,s,t,u,v,w\}$ and the left and right products given in Table \ref{dioctoniontable}.  
The involution is defined by the following equations:
  \[
  p^\ast =  q, \;\;
  q^\ast =  p, \;\;
  r^\ast = -r, \;\;
  s^\ast = -s, \;\;
  t^\ast = -t, \;\;
  u^\ast = -u, \;\;
  v^\ast = -v, \;\;
  w^\ast = -w.
  \]
This dialgebra is neither commutative nor associative, but it is alternative.
The Leibniz bracket makes the underlying vector space into a nonassociative algebra with 
structure constants given in Table \ref{dimalcevtable}.
  
  \begin{table}[h]
  \[
  \begin{array}{l|rrrrrrrr}
  \dashv & p & q & r & s & t & u & v & w \\
  \midrule
  p &  p &  p &  s &  s &  u &  u &  v &  v \\   
  q &  q &  q &  r &  r &  t &  t &  w &  w \\
  r &  r &  r & -q & -q & -v & -v &  u &  u \\
  s &  s &  s & -p & -p & -w & -w &  t &  t \\
  t &  t &  t &  v &  v & -q & -q & -s & -s \\
  u &  u &  u &  w &  w & -p & -p & -r & -r \\
  v &  v &  v & -t & -t &  r &  r & -p & -p \\
  w &  w &  w & -u & -u &  s &  s & -q & -q
  \end{array}
  \]
  \bigskip
  \[
  \begin{array}{l|rrrrrrrr}
  \vdash & p & q & r & s & t & u & v & w \\
  \midrule
  p &  p &  q &  r &  s &  t &  u &  v &  w \\
  q &  p &  q &  r &  s &  t &  u &  v &  w \\
  r &  r &  s & -p & -q & -v & -w &  t &  u \\
  s &  r &  s & -p & -q & -v & -w &  t &  u \\
  t &  t &  u &  v &  w & -p & -q & -r & -s \\
  u &  t &  u &  v &  w & -p & -q & -r & -s \\
  v &  w &  v & -u & -t &  s &  r & -q & -p \\
  w &  w &  v & -u & -t &  s &  r & -q & -p
  \end{array}
  \]
  \medskip
  \caption{Dialgebra analogue of the octonions}
  \label{dioctoniontable}
  \end{table}
 
  \begin{table}[h]
  \[
  \begin{array}{c|cccccccc}
  [-,-] & p & q & r & s & t & u & v & w \\
  \midrule
  p &
  0 &
  0 &
  -  r +  s &  
  -  r +  s &  
  -  t +  u &  
  -  t +  u &  
     v -  w &  
     v -  w 
  \\
  q &  
  0 &
  0 &  
     r -  s &  
     r -  s &  
     t -  u &  
     t -  u &  
  -  v +  w &  
  -  v +  w 
  \\
  r &  
  0 &
  0 &  
     p -  q &  
     p -  q &  
  -2 v &  
  -2 v &  
   2 u &  
   2 u 
  \\
  s &
  0 &
  0 &  
  -  p +  q &  
  -  p +  q &  
  -2 w &  
  -2 w &  
   2 t &  
   2 t 
  \\  
  t &
  0 &
  0 &  
   2 v &  
   2 v &  
     p -  q &  
     p -  q &  
  -2 s &  
  -2 s 
  \\  
  u & 
  0 &
  0 &  
   2 w &  
   2 w &  
  -  p +  q &  
  -  p +  q &  
  -2 r &  
  -2 r 
  \\  
  v &
  0 &
  0 &  
  -2 t &  
  -2 t &  
   2 r &  
   2 r &  
  -  p +  q &  
  -  p +  q 
  \\  
  w &
  0 &
  0 &  
  -2 u &  
  -2 u &  
   2 s &  
   2 s &  
     p -  q &  
     p -  q  
  \end{array}
  \]
  \medskip
  \caption{Leibniz bracket on 8-dimensional alternative dialgebra}
  \label{dimalcevtable}
  \end{table}


\section*{Acknowledgements}

R. Felipe-Sosa thanks the CIMAT in Guanajuato for its support and hospitality during July and August 2011. 
R. Felipe was supported by CONACyT grant 106923; he thanks M. R. Bremner and J. S\'anchez-Ortega for 
introducing him to the KP algorithm.
J. S\'anchez-Ortega was partially supported by the MINECO through project MTM2010-15223
and the Junta de Andaluc\'ia and Fondos FEDER through projects FQM 336 and FQM 02467.
M. R. Bremner thanks NSERC for financial support, and the ICIMAF in Havana for its hospitality during 
December 2011.
The last four authors thank BIRS, the Banff International Research Station, for the opportunity to work 
together during a Research in Teams program from 29 April to 6 May 2012.


\end{document}